\documentclass[12pt]{amsart}

\usepackage{amstext}
\usepackage{amsthm}
\usepackage{amsmath}
\usepackage{latexsym}
\usepackage{amsfonts}
\usepackage{mathtools}
\usepackage{enumerate}
\usepackage{geometry}
\usepackage{blindtext}
\usepackage{incgraph,tikz}
\usepackage[utf8]{inputenc}
\usepackage{bbold}
\usepackage{nameref}
\usepackage[normalem]{ulem}
\usepackage{wasysym}
\usepackage{csquotes}
\usepackage{anyfontsize}
\usepackage{hyperref}

\newtheorem*{problem*}{Problem}
\newtheorem{theorem}{Theorem}[section]
\newtheorem*{question*}{Question}
\newtheorem{lemma}[theorem]{Lemma}
\newtheorem{proposition}[theorem]{Proposition}

\theoremstyle{definition}\newtheorem{remark}[theorem]{Remark}
\theoremstyle{definition}
\theoremstyle{definition}
\numberwithin{equation}{section}

\allowdisplaybreaks

\newcommand\Hol{{\mathrm{Hol}(\mathbb D)}}

\newcommand\cD{{\mathcal D}}

\newcommand\CC{{\mathbb C}}
\newcommand\RR{{\mathbb R}}
\newcommand\DD{{\mathbb D}}
\newcommand\TT{{\mathbb T}}
\newcommand\DDD{{\mathcal D}}

\title[Cantor sets and cyclicity]{Riesz $\alpha$-capacity of Cantor sets and cyclicity in Dirichlet-type spaces}
\author[D. Vavitsas]{Dimitrios Vavitsas}
\address{School of Mathematics (Zhuhai), Sun Yat-Sen University, Zhuhai, Guangdong, 519082, P. R.
China}
\email{vavitsas@mail.sysu.edu.cn}
\author[J. Wu]{Jujie Wu}
\address{School of Mathematics (Zhuhai), Sun Yat-Sen University, Zhuhai, Guangdong, 519082, P. R.
China}
\email{wujj86@mail.sysu.edu.cn}
\author[K. Zarvalis]{Konstantinos Zarvalis}
\address{Department of Mathematics, Aristotle University of Thessaloniki, 54124, Thessaloniki, Greece}
\email{zarkonath@math.auth.gr}


\begin{document}
\maketitle
\begin{abstract}
We examine the threshold of the cyclicity for functions in Dirichlet-type spaces $\mathcal{D}_\alpha$, $\alpha\in(0,1]$. Given a fixed $\alpha^*\in(0,1]$, we construct a holomorphic function $f\in D_{\alpha^*}$ which is cyclic in $\mathcal{D}_\alpha$ for all $ \alpha<\alpha^*$, but fails to be cyclic in $\mathcal{D}_{\alpha^*}$. This function serves as a counterexample to the persistence of cyclicity at the critical index $\alpha^*$. Throughout the construction process, we work with generalized Cantor sets and study their Riesz $\alpha$-capacity. 
\end{abstract}

\section{Introduction}

This article explores the cyclicity of functions with respect to the \textit{Dirichlet-type spaces} $\mathcal{D}_\alpha$, $\alpha\in(0,1]$, of the unit disc $\DD$ in the complex plane $\CC$. Our objective is to discover whether cyclicity in these spaces is ``closed'' in the following sense: given a fixed $\alpha^* \in (0,1]$, does the cyclicity of $f$ in $\DDD_\alpha$, for all $\alpha < \alpha^*$, imply the cyclicity of $f$ in $\DDD_{\alpha^*}$? We will answer this question in the negative by constructing explicit counterexamples. 

In order to produce our counterexamples, we shall work beforehand with \textit{generalized Cantor sets} of the unit circle $\TT\vcentcolon=\partial\DD$. By the term generalized Cantor set, we mean Cantor sets produced as follows: Take the unit circle and remove an open subarc with $-1$ in the middle, so that two closed arcs of equal length with common endpoint $1$ remain. Then, remove two open subarcs from their middles, so that four closed arcs remain (two of them have common endpoint $1$). Continue this process inductively, and what remains is a generalized Cantor set in $\TT$. Such sets are known to be Carleson sets. 

We are going to study generalized Cantor sets with regard to their \textit{Riesz} $\alpha$\textit{-capacity}, $\alpha\in(0,1]$ (the case $\alpha=1$ yields the well-known \textit{logarithmic capacity}). Riesz $\alpha$-capacity is a tool that gauges the size of subsets of the unit circle, but essentially differs from the Lebesgue measure. As a matter of fact, it provides a more delicate approach in measuring sets, since there exist compact subsets of the unit circle whose Lebesgue measure is zero, yet they are of positive Riesz $\alpha$-capacity for certain parameters $\alpha\in(0,1]$.  

Our first objective is to make explicit constructions of generalized Cantor sets satisfying prescribed properties with respect to their Riesz capacities. More specifically, given $\alpha^*\in(0,1]$, we will create such a set $E_{\alpha^*}$ whose Riesz $\alpha$-capacity vanishes for all $\alpha\in (0,\alpha^*)$, whereas its $\alpha^*$-Riesz capacity is positive. In addition, whenever $\alpha^*\in(0,1)$ we will proceed to a similar construction of a set $F_{\alpha^*}$ whose Riesz $\alpha$-capacity vanishes if and only if $\alpha\in(0,\alpha^*]$.

Explicit constructions of such Carleson sets are of interest in their own right. However, due to the strong ties betweeen potential theoretic tools and functional analysis, we may utilize our generalized Cantor sets to produce results relative to spaces of holomorphic functions. As a matter of fact, the notion of Riesz $\alpha$-capacity is closely related to the cyclicity of functions in the Dirichlet-type spaces of the unit disc. This important topic has been at the forefront of many advancements, from the seminal paper \cite{Brown-Shields} of Brown and Shields, to a wide variety of high-quality research conducted in the past decades; see \cite{Brown-Cohn,Primer,EFKR,EFKR2,EFKR*} to name a few. 

The Dirichlet-type spaces $\cD_\alpha$, $\alpha\in \RR$, form a family consisting of Hilbert spaces of holomorphic functions in the unit disc. For the values $\alpha>1$, the space $\cD_\alpha$ is an algebra. On the other hand, whenever $\alpha\leq 1$, $\mathcal{D}_\alpha$ is a Banach space which is not an algebra, and its cyclic vectors play a role similar to that played by invertible elements in algebras.  Classical Hilbert spaces such as the Bergman, the Hardy and the Dirichlet spaces are all special instances of Dirichlet-type spaces; in particular, in terms of equivalent norms, $\cD_{-1}$ is the Bergman space, $\cD_0$ is the Hardy space, and $\cD_1$ is the Dirichlet space (denoted by $\DDD$). For a presentation of the theory concerning this family of weighted Dirichlet spaces, we refer the interested reader to \cite{Aleman1}, \cite{Brown-Shields} and \cite[Chapter 1]{Primer}. More information and concrete definitions follow in Section \ref{sec:preliminaries}.

We will mainly deal with Dirichlet-type spaces $\mathcal{D}_\alpha$ where the parameter lies in the interval $(0,1]$. These spaces provide a natural scale of Banach spaces linking the classical Hardy space to the classical Dirichlet space. Let us point out that in the Hardy space ($\alpha=0$), a complete characterization of cyclicity was achieved by Beurling; see \cite{Beurling}. In fact, a function $f$ is cyclic in $\mathcal{D}_0$ if and only if it is outer. The characterization of cyclic in $\cD_\alpha$ functions, for $\alpha\in (0,1]$, remains an open problem.

It can be easily deduced that if $\alpha_1,\alpha_2\in\RR$ with $\alpha_1<\alpha_2$ and $f\in\mathcal{D}_{\alpha_2}$ is cyclic in $\mathcal{D}_{\alpha_2}$, then it necessarily belongs to and is cyclic in $\mathcal{D}_{\alpha_1}$ as well. The exact converse trivially does not hold. So, it is natural to wonder whether the following statement holds:
\begin{problem*}\phantomsection \label{Problem}
Fix $\alpha^*\in (0,1]$ and suppose that $f$ is cyclic in $\mathcal{D}_\alpha$, for all $\alpha<\alpha^*$. Is $f$ necessarily cyclic in $\mathcal{D}_{\alpha^*}$?    
\end{problem*}

This question constitutes the main motivation of our work. We are going to answer it in the negative by providing two different explicit counterexamples. Given a function $f$ as in the statement above, let us denote its \textit{zero set} by $\mathcal{Z}(f)=\{z\in\DD\vcentcolon f(z)=0\}$. A wide variety of published work points to the fact that the cyclicity of $f$ in the spaces $\mathcal{D}_\alpha$ is inextricably related to the size of $\mathcal{Z}(f^*)$ with respect to its Riesz $\alpha$-capacity, where $f^*$ denotes the radial limit function of $f$. Recall that by Fatou's Theorem, $f^*$ is defined almost everywhere on $\TT$.  

With this in mind, we will construct, as we mentioned, a generalized Cantor set $E_{\alpha^*}\subset \TT$ that is a Carleson set and whose Riesz $\alpha$-capacity vanishes for all $\alpha\in (0,\alpha^*)$, whereas its Riesz $\alpha^*$-capacity is positive. This will in turn lead to the construction of two fundamentally different functions whose zero set coincides with $E_{\alpha^*}$ and which act as our desired counterexamples.

The first ``threshold'' function is constructed exclusively through the fact that $E_{\alpha^*}$ is a Carleson set and is independent of the actual values of its Riesz capacities. Surprisingly, the resulting function $f$ belongs to $\cD_{\alpha^*}$, can be created to be as smooth as desired on the closed unit disc, but will still not be cyclic in $\cD_{\alpha^*}$. An important takeway is that cyclicity is not a ``closed'' notion even when the function we investigate satisfies nice differentiability conditions.

On the other hand, the second construction relies profoundly on the vanishing of the Riesz $\alpha$-capacity of $E_{\alpha^*}$ for $\alpha\in (0,\alpha^*)$. In fact, the very definition of the function that will play the role of our counterexample is based on the Riesz capacities of neighborhoods of $E_{\alpha^*}$ with respect to the topology of $\TT$. The resulting function $f$ will belong to the space $\cD_{\alpha}$, for all $\alpha<\alpha^*$, will actually be continuous up to the unit circle, and will be cyclic in $\cD_{\alpha}$, for all $\alpha<\alpha^*$. Nonetheless, either $f\notin\cD_{\alpha^*}$ (and by extension it is obviously not cyclic in it), or $f\in\cD_{\alpha^*}$ but is not cyclic in it. 

In both cases, the non-cyclicity of the critical function in $\cD_{\alpha^*}$ hinges on the positivity of the Riesz $\alpha^*$-capacity of $\mathcal{Z}(f^*)$. Let us note that, in case $\alpha^*=1$, this question is related to the celebrated Brown--Shields Conjecture \cite[Question 12]{Brown-Shields} concerning the characterization of cyclic functions in $\cD$ because, in some sense, our constructions support its validity.

On another note, the interval of the parameters where a polynomial is cyclic in the unit disc setting, in the two-dimensional unit ball setting and in the bidisc setting is closed from above; see \cite{Kosinski-Sola1,Bergvist,Brown-Shields,Kosinski-Sola2, Vavitsas1, Vavitsas2} for polynomials and a recent discussion with explicit examples in \cite{Pouriya}. Nevertheless, it is natural to seek a counterexample to this phenomenon in the case of general holomorphic functions.

The article is structured as follows: In Section \ref{sec:preliminaries} we proceed to a consice presentation of the basic notions and tools we are going to use throughout the course of this work. Then, in Section \ref{sec:cantor} we will construct the suitable generalized Cantor sets required for solving the \hyperref[Problem]{Problem}. Next, in the final two sections of the paper, we will state and prove necessary lemmas, and we will construct our critical functions.

\section{Preliminaries}\label{sec:preliminaries}

\subsection{Riesz $\alpha$-energy and Riesz $\alpha$-capacity}
We commence with a very useful tool from potential theory. As we have already mentioned in the Introduction, we will make use of the Riesz $\alpha$-capacity of sets on the unit circle. To properly define it, let $\alpha\in(0,1]$ and consider the non-negative \textit{kernel} $K_\alpha:(0,+\infty)\rightarrow [0,+\infty)$ given by
\begin{equation}\label{eq:kernel}
K_\alpha(t)\vcentcolon=
\begin{cases}
    \log^+(2/t), \quad\quad\hspace{0.1cm}\text{for }\alpha=1\\
    t^{\alpha-1}, \quad\quad\quad\quad\hspace{0.2cm}\text{for } \alpha\in (0,1),
\end{cases}
\end{equation}
where $\log^+(x)\vcentcolon=\max\{0,\log x\}$, for $x>0$. Let $\mu$ be a Borel probability measure supported on a compact set $E\subset \TT$. Its \textit{Riesz} $\alpha$\textit{-energy} is defined through
\begin{equation}\label{eq:energy}
I_\alpha[\mu]\vcentcolon=\int_{\TT}\int_{\TT}K_\alpha(|\zeta-\eta|)d\mu(\eta)d\mu(\zeta)=\int_{\TT}K_{\alpha,\mu}(\zeta)d\mu(\zeta),
\end{equation}
where
\begin{equation}\label{eq:potential}
    K_{\alpha,\mu}(\zeta)\vcentcolon=\int_{\TT}K_\alpha(|\zeta-\eta|)d\mu(\eta)
\end{equation}
is called the \textit{Riesz} $\alpha$\textit{-potential} of $\mu$. Then, the \textit{Riesz} $\alpha$\textit{-capacity} of $E$ is given by
\begin{equation*}
\mathrm{cap}_{\alpha}(E):=\inf\{I_\alpha[\mu]:\mu\in \mathcal{P}(E)\}^{-1},
\end{equation*}
where $\mathcal{P}(E)$ denotes the set of all Borel probability measures supported on $E$. Note that $\mathrm{cap}_\alpha(E)>0$ yields $I_\alpha[\mu]<+\infty$ for at least one measure $\mu\in \mathcal{P}(E)$. On the other hand, $\mathrm{cap}_\alpha(E)=0$ if and only if  $I_\alpha[\mu]=+\infty$ for all $\mu\in \mathcal{P}(E)$. In case $\alpha=1$, then $I_\alpha[\mu]$ and $\mathrm{cap}_\alpha(E)$ are called the \textit{logarithmic energy} of $\mu$ and the \textit{logarithmic capacity} of $E$, respectively. The definition of capacity may be extended to general sets (not necessarily compact) on the unit circle by means of the so-called \textit{inner capacity}. Indeed, for any $E\subseteq\TT$ and any $\alpha\in(0,1]$, its inner capacity is given by
\begin{equation*}
    \mathrm{cap}_\alpha(E)\vcentcolon=\sup\{\mathrm{cap}_\alpha(F) \vcentcolon \, F\subset E, \, F \text{ compact}\}.
\end{equation*}
There exists a similar definition for the \textit{outer capacity}, but for Borel sets, the two definitions coincide due to Choquet's Theorem; see \cite[Definition 2.1.8]{Primer} and \cite[Theorem 5.1.2]{Ransford}. From now on, for any Borel subset of the unit circle we are going to just use the term capacity for the inner capacity as well.

If $E\subset \TT$ is a compact set and $\mathrm{cap}_{\alpha}(E)>0$, then an \emph{equilibrium measure} for $E$ is a measure $\nu\in\mathcal{P}(E)$ such that $I_\alpha(\mu)\geq I_\alpha[\nu]$, for all $\mu\in\mathcal{P}(E)$. Evidently, if $\nu$ is an equilibrium measure, the definition of capacity dictates that $\mathrm{cap}_{\alpha}(E)=1/I_\alpha[\nu]$. In particular, any compact set $E\subset \TT$ has an equilibrium measure, say $\nu$, which satisfies the following properties:
\begin{enumerate}
    \item[(i)] $K_{\alpha,\nu}(\zeta)\leq I_{\alpha}[\nu]$, for all $\zeta\in \TT$;
    \item[(ii)] $K_{\alpha,\nu}(\zeta)=I_\alpha[\nu],$ $\alpha$\textit{-almost everywhere} on $\TT$, that is, on the whole $\TT$ except possibly for a set whose Riesz $\alpha$-capacity is zero.
\end{enumerate}

The properties above constitute a slightly different version of Frostman's Theorem. For a profound account of potential theory and the rich theory of capacity, we point the interested reader to \cite{Frostman} and the books \cite{Carleson, Primer, Ransford}.

\subsection{Dirichlet-type spaces and cyclicity}\label{subsec:dirichlet}
Let $\mathrm{Hol}(\DD)$ denote the class of holomorphic in $\DD$ functions and by $A^k(\DD)$ the subset of $\mathrm{Hol}(\DD)$ consisting of all holomorphic functions with continuous derivatives of order less than or equal to $k$ on the closure of the unit disc. In the case $k=0$, we write $A(\DD)\vcentcolon=A^0(\DD)$ for the usual disc algebra. Furthermore, we use the notation $\mathcal{O}(\overline{\DD})$ for the class of all functions holomorphic in a neighborhood of $\overline{\DD}$. Given $\alpha\in\RR$, a function $f\in \Hol$ with power series expansion
\begin{equation*}
f(z)=\sum_{k=0}^{+\infty}a_kz^k, \quad z\in \DD,
\end{equation*}
is said to belong to the \textit{Dirichlet-type space} $\cD_\alpha$ provided that 
\begin{equation}\label{eq:norm}
\|f\|^2_\alpha\vcentcolon=\sum_{k=0}^{+\infty}(k+1)^\alpha|a_k|^2<+\infty.
\end{equation}
If $f\in\cD_\alpha$, a direct consequence of \eqref{eq:norm} is that $zf\in\cD_\alpha$ as well. In addition, \eqref{eq:norm} shows that the family of Dirichlet-type spaces is decreasing in the sense that $\cD_{\alpha_1}\subset \cD_{\alpha_2}$, for any pair of real numbers $\alpha_1>\alpha_2$, since evidently $\|\cdot\|_{\alpha_1}\ge\|\cdot\|_{\alpha_2}$. 

Cyclicity in Dirichlet-type spaces is usually defined with respect to the \textit{shift operator} given by $T:\cD_\alpha\rightarrow \cD_\alpha$, $T(f)\vcentcolon=zf$, which is a well-defined bounded linear transformation. We say that a function $f\in\cD_\alpha$, $\alpha\in \RR$, is a \textit{cyclic} vector for $\mathcal{D}_\alpha$ if 
\begin{equation*}
[f]_\alpha\vcentcolon=\mathrm{clos}\text{ }\mathrm{span}\{z^nf\vcentcolon n=0,1,...\},
\end{equation*}
where the closure is taken with respect to the norm, coincides with the whole space $\cD_\alpha$. The space $[f]_\alpha$ is closed and invariant under the action of the shift operator.

\begin{remark}\label{equiv cyclic}
It is worth noting that $f$ is cyclic in $\cD_\alpha$ if and only if there exists a sequence of polynomials $(p_n)_{n\in \mathbb N}$ such that $p_nf\rightarrow 1$ in the $\cD_\alpha$-norm. The same holds in case the convergence occurs weakly instead of in the $\cD_\alpha$-norm. Equivalently, $f$ is cyclic if and only if $1\in [f]_\alpha$. We refer to \cite[Proposition 5]{Brown-Shields} for more details.
\end{remark}

Additionally, a function $h$ of $\DD$ is called a \textit{multiplier} of $\cD_\alpha$ provided $hf\in\cD_\alpha$ whenever $f\in\cD_\alpha$. It is known by \cite[Proposition 7]{Brown-Shields} that multipliers leave any subspace $[f]_\alpha$ invariant. In other words, if $h$ is a multiplier and $g\in[f]_\alpha$, for some $f\in\cD_\alpha$, $\alpha\in\RR$, then $hg\in[f]_\alpha$ as well. We will need this property of multipliers towards the end of our investigation. For more information regarding Dirichlet-type spaces and the corresponding cyclicity, we refer to \cite{Aleman1, Beurling, Brown-Shields, Primer, Mironov, Richter1, Richter-Shields, Richter, Ross, Sampat, Taylor}.

From now on, we mostly restrict ourselves to the Dirichlet-type spaces $\cD_\alpha$, for $\alpha\in(0,1]$. Clearly $\|\cdot\|_\alpha$ defines a norm and hence the family $(\cD_\alpha,\|\cdot\|_\alpha)_{\alpha\in (0,1]}$ consists of Banach spaces of holomorphic functions.  Moreover, the membership of a function in these spaces can be identified through the computation of an integral. Indeed, for $\alpha\in(0,1]$, $f\in\cD_\alpha$ if and only if
\begin{equation}\label{eq:integral}
    \int_{\DD}|f'(z)|^2(1-|z|^2)^{1-\alpha}dA(z)<+\infty,
\end{equation}
where $dA$ denotes the area Lebesgue measure. Fix such an $\alpha$ and consider a function $f\in\cD_\alpha$. Then, its radial limit function $f^*$ is defined through $f^*(\zeta)\vcentcolon=\lim_{r\to1^-}f(r\zeta)$, for $\zeta\in\TT$ except possibly for a set of zero Riesz $\alpha$-capacity; see \cite[\S V, Theorem 3]{Carleson} which provides a sharper result than Fatou's Theorem. Besides, even if the function $f$ extends continuously to $\TT$, we will maintain the notation $f^*\vcentcolon=f_{|_{\TT}}$. As we said before, the cyclicity of $f$ in $\cD_\alpha$ strongly depends on the Riesz $\alpha$-capacity of the zero set $\mathcal{Z}(f^*)$. A first remarkable result, resulting from the groundbreaking work of Brown and Shields, and demonstrating this dependence is the following:

\begin{theorem}[{cf. \cite[Corollary 1, Theorem 5]{Brown-Shields},\cite[Theorem 1.1]{EFKR}}] \label{thm: efkr 1}
    Let $\alpha\in(0,1]$ and suppose that $f\in\mathcal{\cD}_\alpha$ is cyclic in $\cD_{\alpha}$. Then, $f$ is an outer function and $\mathrm{cap}_{\alpha}(\mathcal{Z}(f^*))=0$.
\end{theorem}
A partial converse statement is provided through \cite[Theorem 1.2]{EFKR} which we state below. We also refer to \cite[Corollary 1.2]{EFKR*} for the classical Dirichlet space $\cD$.
\begin{theorem}\label{thm: efkr 2}
    Let $\alpha\in(0,1]$ and $f\in\mathcal{\cD}_\alpha$. Suppose that
    \begin{enumerate}
        \item[\textup{(i)}] $f$ is an outer function;
        \item[\textup{(ii)}] $|f|$ extends continuously to $\overline{\DD}$;
        \item[\textup{(iii)}] $\mathcal{Z}(f^*)$ is contained in a generalized Cantor set of zero Riesz $\alpha$-capacity.
    \end{enumerate}
    Then, $f$ is cyclic in $\cD_\alpha$.
\end{theorem}
The two theorems above and this interplay between cyclicity and Riesz $\alpha$-capacity will be crucial in the sequel.

\section{Generalized Cantor sets}\label{sec:cantor}

We are ready to proceed to the main body of the article and the first objective of our work. In this section, we will describe the general process of constructing generalized Cantor sets, and then we will construct the explicit examples that we desire.

Since our Cantor sets will be eventually used as zero sets of radial limit functions, we want them to be subsets of the unit circle. Nevertheless, we will define them, as usual, on the rectilinear segment $[0,1]$, and then take their image through the transformation $T(\theta)=e^{i2\pi\theta}$, $\theta\in[0,1]$.

We mostly follow \cite[Section 2.3]{Primer}. To construct a generalized Cantor set, in the first step, we remove from $E_0^1\vcentcolon=[0,1]$ a middle open segment of length $e_0\in(0,1)$ so that two closed segments of equal length $d_1$ remain. We denote them by $E_1^1$ and $E_1^2$. Then, from each set $E_1^j$, $j=1,2$, we remove a middle open segment of length $e_1\in(0,d_1)$, so that in total four closed segments of equal length $d_2$ remain, which we denote by $E_2^j$, $j=1,2,3,4$. In general, in the $n$-th step of the construction, from each of the existing sets $E_{n-1}^j$, $j=1,\dots,2^{n-1}$ (who have equal length $d_{n-1}$), we remove a middle open segment of length $e_{n-1}\in(0,d_{n-1})$, so that the sets $E_n^j$, $j=1,\dots,2^n$, remain and they all have equal length $d_n$. Set $E_n=\cup_{j=1}^{2^n}E_n^j$, $n\in\mathbb{N}\cup\{0\}$. Then, the \textit{generalized (middle third) Cantor set} associated to the sequence $\{E_n\}_{n\ge0}$ is the set 
\begin{equation*}
    E\vcentcolon=\bigcap_{n=0}^{+\infty}E_n=\bigcap_{n=0}^{+\infty}\bigcup_{j=1}^{2^n}E_n^j. 
\end{equation*}
Setting $d_0=1$ and due to the way we constructed the Cantor set $E$, we may see that $d_n=\mathrm{diam}E_n^j$, for all $n\in\mathbb{N}\cup\{0\}$, for any $j=1,\dots,2^n$. In addition, we may write $e_n=\min\{\mathrm{dist}(E_{n+1}^j,E_{n+1}^k)\vcentcolon j\ne k, \,E_{n+1}^j\cup E_{n+1}^k\subseteq E_n^i\text{ for some } i\in1,\dots,2^n\}$, for each $n\in\mathbb{N}$. By construction, the two sequences $\{d_n\}$ and $\{e_n\}$ are related through the formula
\begin{equation}\label{eq:cantor sequences}
  e_n=d_n-2d_{n+1}, \quad n\in\mathbb{N}\cup\{0\}.
\end{equation}

\begin{remark}\label{rem:cantor}
Clearly, the set $E$ is linked to the sequences $\{d_n\}$ and $\{e_n\}$. As a matter of fact, picking two such sequences, following the process described above leads to the construction of a generalized Cantor set in $[0,1]$. Of course, the sequences need to fulfill certain requirements. Firstly, we need $d_0=1$ and $\{d_n\}_{n\in\mathbb{N}}\subset(0,1)$. In addition, $d_n$ need to decrease (not necessarily strictly) to $0$. Finally, the existence of the positive numbers $e_n$, $n\in\mathbb{N}\cup\{0\}$, means that $d_n-2d_{n+1}>0$, for all $n\in\mathbb{N}\cup\{0\}$.
\end{remark}

Having our set on the segment $[0,1]$ ready, we can move to the unit circle. The set $T(E)$ associated to the sequence $\{T(E_n)\}=\{\cup_{j=1}^{2^n}T(E_n^j)\}$ is a \textit{generalized Cantor set} on $\TT$. In similar fashion to the case of the real Cantor set above, set $\delta_n\vcentcolon= \mathrm{diam}T(E_n^j)$, for any $j=1,\dots,2^n$, for all $n\in\mathbb{N}\cup\{0\}$. Note that since our generalized Cantor sets are produced by removing middle third open segments in each step, for each $n$, the diameters of the sets $T(E_n^1),\dots,T(E_n^{2^n})$, are equal. Moreover, for each $n$, set
\begin{equation*}
    \epsilon_n \vcentcolon= \min\{\mathrm{dist}(T(E_{n+1}^j),T(E_{n+1}^k))\vcentcolon j\ne k, \, E_{n+1}^j,E_{n+1}^k\subseteq E_n^i \text{ for some }i=1,\dots,2^n\}.
\end{equation*}
Evidently, the numbers $\delta_n$ and $\epsilon_n$ are related to the numbers $d_n$ and $e_n$, respectively. We prove a first auxiliary lemma demonstrating this relation.

\begin{lemma}\label{lem:relation of Cantor sets}
    Let $E$ be a generalized Cantor set on the segment $[0,1]$ associated to $\{E_n\}$ and $T\vcentcolon[0,1]\to\TT$ the transformation with $T(\theta)=e^{i2\pi\theta}$. Consider $d_n,e_n,\delta_n,\epsilon_n$ to be as above. Then:
    \begin{enumerate}
        \item[\textup{(a)}] $|T(\theta_2)-T(\theta_1)|=2\sin(\pi|\theta_2-\theta_1|)$, for all $\theta_1,\theta_2\in[0,1]$.
        \item[\textup{(b)}] There exists an absolute constant $c>1$ such that $\frac{1}{c}d_n \le \delta_n \le c d_n$, for all $n\in\mathbb{N}$.
        \item[\textup{(c)}] There exists an absolute constant $c'>1$ such that $\frac{1}{c'}e_n \le \epsilon_n \le c' e_n$, for all $n\in\mathbb{N}$.
    \end{enumerate}
\end{lemma}
\begin{proof}
    (a) Fix $\theta_1,\theta_2\in[0,1]$. Through trivial calculations, $|T(\theta_2)-T(\theta_1)|=|e^{i2\pi\theta_1}-e^{i2\pi\theta_2}|=\sqrt{2-2\cos(2\pi\theta_1)\cos(2\pi\theta_2)-2\sin(2\pi\theta_1)\sin(2\pi\theta_2)}$. The trigonometric identities $\cos x \cos y+\sin x\sin y=\cos(x-y)$ and $2\sin^2(x/2)=1-\cos x$, provide the desired outcome. 

    (b) Fix $n\in\mathbb{N}$. As we said before, $d_n=\mathrm{diam}E_n^j$ and $\delta_n=\mathrm{diam}T(E_n^j)$, for any choice of $j=1,\dots,2^n$. Pick such a $j$ in random and let $\alpha_n$ and $\beta_n$ be the endpoints of the closed segment $E_n^j$. Obviously, $d_n=|\alpha_n-\beta_n|$, while $\delta_n=|T(\alpha_n)-T(\beta_n)|$, with the second equality holding since the set $E_n^j$ is necessarily smaller than a half-circle. So, due to (a), we have $\delta_n=2\sin(\pi d_n)$. But in order to have the Cantor set $E$, it is evident that $\lim_{n\to+\infty}d_n=0$. As a result,
    \begin{equation*}
        \lim_{n\to+\infty}\frac{\delta_n}{d_n}=\lim_{n\to+\infty}\frac{2\sin(\pi d_n)}{d_n}=2\pi.
    \end{equation*}
    This last limit proves the existence of the desired absolute constant $c$. Note that for $n=0$, the non-univalence of $T$ yields $\delta_0=0$, while $d_n>0$. Excluding this case, the constant exists for all $n\in\mathbb{N}$.

    (c) The proof follows the same procedure as in (b). So we omit it, for the sake of avoiding repetition.
\end{proof}

With Lemma \ref{lem:relation of Cantor sets} in our quiver, we can state a variant of \cite[Theorem 2.3.5]{Primer} which allows us to estimate the Riesz $\alpha$-capacity, $\alpha\in(0,1]$, of a generalized Cantor type set $E$ on $\TT$, just by observing its pre-image $T^{-1}(E)$ on $[0,1]$. This is because by (b) and (c) of the previous lemma, the sequences $\{d_n\}$, $\{e_n\}$ and $\{\delta_n\}$, $\{\epsilon_n\}$, respectively, are Lipschitz equivalent.

\begin{lemma}[cf. {\cite[Theorem 2.3.5]{Primer}}]\label{lem:capacity of Cantor}
    Let $E$ be a generalized Cantor set on the segment $[0,1]$ and $T\vcentcolon[0,1]\to\TT$ the transformation $T(\theta)=e^{i2\pi\theta}$. Consider $d_n$ and $e_n$ to be as above. Then, for any $\alpha\in(0,1]$,
    \begin{equation}\label{eq:capacity Cantor}
        \sum_{n=0}^{+\infty}\frac{K_\alpha(d_n)}{2^{n+1}} \le \frac{1}{\mathrm{cap}_\alpha(T(E))} \le \sum_{n=0}^{+\infty}\frac{K_\alpha(e_n)}{2^n},
    \end{equation}
    where $K_\alpha:(0,+\infty)\to[0,+\infty)$ is the kernel defined in \eqref{eq:kernel}.
\end{lemma}
Lemma \ref{lem:capacity of Cantor} will play a crucial role in our aim of constructing appropriate Cantor sets satisfying specific capacity requirements. Its utility hinges on the fact that if the left-hand side of \eqref{eq:capacity Cantor} diverges, then $\mathrm{cap}_\alpha(T(E))=0$, whereas if the corresponding right-hand side converges, then $\mathrm{cap}_\alpha(T(E))>0$. Note that the capacity of $T(E)\subset\TT$ is evaluated in terms of the properties of $E\subset[0,1]$. Another viable option would have been to define capacity in general, even for subsets of $[0,1]$ and then correlate the positivity of the Riesz $\alpha$-capacity of a set in $[0,1]$ with that of its image in $\TT$; see \cite[Proposition 2]{Vavitsas-Zarvalis} for a relevant application.

We now construct our first generalized Cantor set. We exclude the case $\alpha^*=1$ since the corresponding result is trivial.

\begin{proposition}\label{prop:cantor closed}
    Fix $\alpha^*\in(0,1)$. There exists a generalized Cantor set $F\subset\TT$ such that $\mathrm{cap}_\alpha(F)=0$, for all $\alpha\in(0,\alpha^*]$, whereas $\mathrm{cap}_\alpha(F)>0$, for all $\alpha\in(\alpha^*,1]$.
\end{proposition}
\begin{proof}
In order to achieve our objective, we will first construct the corresponding generalized Cantor set of the segment $[0,1]$. For $n\in\mathbb{N}\cup\{0\}$ set 
\begin{equation}\label{eq:exercise 1}
   d_n\vcentcolon=\left(2^{-\frac{1}{1-\alpha^*}}\right)^n=2^\frac{n}{\alpha^*-1}. 
\end{equation}
Since $1-\alpha^*\in(0,1)$, the sequence $\{d_n\}$ lies in $(0,1]$ and strictly decreases from $1$ to $0$. On top of that, for each $n\in\mathbb{N}\cup\{0\}$,
\begin{align}\label{eq:exercise 2}
\notag    e_n\vcentcolon&=d_n-2d_{n+1}=2^\frac{n}{\alpha^*-1}-2^{1+\frac{n+1}{\alpha^*-1}}\\
  &=2^\frac{n}{\alpha^*-1}\left(1-2^{1+\frac{1}{\alpha^*-1}}\right)=2^\frac{n}{\alpha^*-1}\left(1-2^{\frac{\alpha^*}{\alpha^*-1}}\right)>0,
\end{align}
where the positivity is due to the fact that $\alpha^*/(\alpha^*-1)<0$. As a result, through the process we described earlier, the sequences $\{d_n\}$ and $\{e_n\}$ define a generalized Cantor set $E$ in $[0,1]$. Write $F\vcentcolon=T(E)$, where as usual, $T(\theta)=e^{i2\pi\theta}$, $\theta\in[0,1]$. We will now use Lemma \ref{lem:capacity of Cantor} to estimate the Riesz $\alpha$-capacity of $F$, for the indices $\alpha\in(0,1]$.

We first deal with the case $\alpha\in(0,\alpha^*]$, so we aim to use the left-hand side of \eqref{eq:capacity Cantor}. Utilizing also the definition in \eqref{eq:kernel}, we have
\begin{align}\label{eq:exercise 3}
 \notag   \sum_{n=0}^{+\infty}\frac{K_\alpha(d_n)}{2^{n+1}}&=\frac{1}{2}\sum_{n=0}^{+\infty}\frac{1}{d_n^{1-\alpha}2^n}=\frac{1}{2}\sum_{n=0}^{+\infty}\frac{1}{2^{-n\frac{1-\alpha}{1-\alpha^*}}2^n}\\
    &=\frac{1}{2}\sum_{n=0}^{+\infty}\frac{1}{2^{n\left(1-\frac{1-\alpha}{1-\alpha^*}\right)}}=\frac{1}{2}\sum_{n=0}^{+\infty}\frac{1}{2^{n\frac{\alpha-\alpha^*}{1-\alpha^*}}},
\end{align}
where we have applied \eqref{eq:exercise 1}. Nevertheless, the exponent of $2$ in the latter sum of \eqref{eq:exercise 3} is non-positive, and hence the sum diverges. By Lemma \ref{lem:capacity of Cantor}, we obtain $\mathrm{cap}_\alpha(F)=0$.

On the other hand, for $\alpha\in(\alpha^*,1)$, we are going to need the right-hand side of \eqref{eq:capacity Cantor}. This time, using \eqref{eq:exercise 2}, we infer
\begin{align*}
    \sum_{n=0}^{+\infty}\frac{K_{\alpha}(e_n)}{2^n}&=\sum_{n=0}^{+\infty}\frac{1}{e_n^{1-\alpha}2^n}=\sum_{n=0}^{+\infty}\frac{1}{2^{-n\frac{1-\alpha}{1-\alpha^*}}\left(1-2^{\frac{\alpha^*}{\alpha^*-1}}\right)^{1-\alpha}2^n}\\
    &=\frac{1}{\left(1-2^{\frac{\alpha^*}{\alpha^*-1}}\right)^{1-\alpha}}\sum_{n=0}^{+\infty}\frac{1}{2^{n\frac{a-\alpha^*}{1-\alpha^*}}}.
\end{align*}
This time, the exponent of $2$ in the latter sum is strictly positive, regardless of $n$, so the sum converges. Thus, Lemma \ref{lem:capacity of Cantor} certifies that $\mathrm{cap}_\alpha(F)>0$.

Finally, for the case $\alpha=1$, following a similar procedure with the logarithmic capacity, we may once again prove that $\mathrm{cap}_1(F)>0$. As a consequence, the set $F$, which by construction is a generalized Cantor set of $\TT$, satisfies the desired conditions.
\end{proof}

\begin{remark}
    The usual Cantor set $C$ on $\TT$ is a well-known example of a set satisfying the condition described by Proposition \ref{prop:cantor closed} for $\alpha^*=1-\log2/\log3$; cf. \cite[p. 27]{Primer}. 
\end{remark}

We proceed to a second, more involved, construction. The differentiating factor compared to the previous Cantor set, is that in the following result, the interval of the parameters $\alpha$ where the capacity of $F$ is zero will be open in $(0,1]$, while in the previous case it was closed. In general, Cantor sets where the aforementioned interval is closed are easier to construct. The difficulty principally lies in finding a set where the interval is actually open in $(0,1]$.

\begin{proposition}\label{construction Cantor 2}
Fix $\alpha^*\in (0,1)$. There exists a generalized Cantor set $F\subset \TT$ such that $\mathrm{cap}_\alpha(F)=0$, for all $\alpha\in (0,\alpha^*)$, whereas $\mathrm{cap}_{\alpha^*}(F)>0$.
\end{proposition}
\begin{proof}
As in the previous proposition, we first construct a generalized Cantor set of the interval $[0,1]$. Consider the sequence $\{\alpha_n\}_{n\ge2}$ with 
    \begin{equation}\label{eq:cantor 1}
        \alpha_n=\frac{\sqrt{n}(1-\alpha^*)}{\sqrt{n}-1}, \quad n\in\mathbb{N},\ n\ge2.
    \end{equation}
    By \eqref{eq:cantor 1}, it is evident that there exists some $n_0\in\mathbb{N}$ such that $\alpha_n\in(0,1)$, for all $n\ge n_0$. Furthermore, $\{\alpha_n\}$ is a decreasing sequence converging to $1-\alpha^*$, as $n\to+\infty$. Next, set $d_0=1$,
    \begin{equation}\label{eq:cantor 2}
        d_n=\left(2^{-\frac{1}{\alpha_n}}\right)^n, \quad n\in\mathbb{N}, \ n\ge n_0,
    \end{equation}
    and pick $d_1,\dots,d_{n_0-1}$, arbitrarily in $(0,1)$, so that the inducing sequence $\{d_n\}_{n\in\mathbb{N}\cup\{0\}}$ is strictly decreasing to $0$. It is easy to check from \eqref{eq:cantor 2} that for $n\ge n_0$, the sequence $\{d_n\}_{n\ge n_0}$ is indeed strictly decreasing, is a subset of $(0,1)$, and $\lim_{n\to+\infty}d_n=0$. Therefore, adding the remaining $n_0$ terms makes the whole construction possible. Now, as described in the beginning of this section, we construct a generalized Cantor set $E=\cap_{n=0}^{+\infty}\cup_{j=1}^{2^n}E_n^j$, where in the $n$-th step we remove middle open segments of length $e_{n-1}$ and closed segments of equal length $d_n$ remain. For the Cantor set $E$ to be well-defined, we need $e_n=d_n-2d_{n+1}>0$, for all $n\in\mathbb{N}\cup\{0\}$; see Remark \ref{rem:cantor}. Clearly, the terms $d_0,\dots,d_{n_0-1}$, can be picked so that this condition holds. For $n\ge n_0$, direct algebraic considerations on formulas \eqref{eq:cantor 1} and \eqref{eq:cantor 2} show that the existence of $a_n<1$ is equivalent to $\alpha^*>\frac{1}{\sqrt{n}}$. On the other hand, $d_n-2d_{n+1}>0$ is equivalent to $\sqrt{n+1}-\sqrt{n}<\alpha^*.$ But since $\alpha^*>\frac{1}{\sqrt{n}}$, for all $n\ge n_0$, the previous inequality holds too. So, the generalized Cantor set $E\subset[0,1]$ is well-defined and in particular, for $n\ge n_0$,
    \begin{equation}\label{eq:cantor 3}
        e_n=\left(2^{-\frac{1}{\alpha_{n}}}\right)^{n}-2\left(2^{-\frac{1}{\alpha_{n+1}}}\right)^{n+1}=\left(2^{-\frac{1}{\alpha_{n}}}\right)^{n}\left[1-2^{1+n\left(\frac{1}{\alpha_{n}}-\frac{1}{\alpha_{n+1}}\right)-\frac{1}{\alpha_{n+1}}}\right].
    \end{equation}
    As usual, with the aid of the transformation $T\vcentcolon[0,1]\to\TT$, where $T(\theta)=e^{i2\pi\theta}$, the image $F\vcentcolon=T(E)$ yields a generalized Cantor set on the unit circle. We will now measure its capacity $\mathrm{cap}_\alpha(F)$, for the indices $\alpha\in(0,\alpha^*]$. First, suppose that $\alpha\in(0,\alpha^*)$. By Lemma \ref{lem:capacity of Cantor}, we have that
    \begin{equation}\label{eq:cantor 4}
        \frac{1}{\mathrm{cap}_\alpha(F)}\ge \sum_{n=0}^{+\infty}\frac{1}{d_n^{1-\alpha}2^{n+1}}\ge\frac{1}{2}\sum_{n=n_0}^{+\infty}\frac{1}{\left(2^{\frac{\alpha-1}{\alpha_n}}\right)^n2^n}=\frac{1}{2}\sum_{n=n_0}^{+\infty}\left(\frac{1}{2^{1+\frac{\alpha-1}{\alpha_n}}}\right)^n.
    \end{equation}
    However, the sequence $\{\alpha_n\}$ converges decreasingly to $1-\alpha^*$. Therefore, and since $\alpha\in(0,\alpha^*)$, there exists $n_1\ge n_0$ so that $\alpha_n<1-\alpha$, for all $n\ge n_1$. This leads to $2^{1+\frac{\alpha-1}{\alpha_n}}<1$, for all $n\ge n_1$. Returning to \eqref{eq:cantor 4}, the latter sum diverges and consequently $\mathrm{cap}_\alpha(F)=0$. The choice of $\alpha\in(0,\alpha^*)$ was arbitrary and hence the desired result holds for all such $\alpha$. We are left with estimating $\mathrm{cap}_{\alpha^*}(F)$. This time, in view of Lemma \ref{lem:capacity of Cantor}
    \begin{equation}\label{eq:cantor 5}
    \frac{1}{\mathrm{cap}_{\alpha^*}(F)} \le \sum_{n=0}^{+\infty}\frac{1}{e_n^{1-\alpha^*}2^n}=c_{n_0}+\sum_{n=n_0}^{+\infty}\frac{1}{2^{\sqrt{n}}\left[1-2^{1+n\left(\frac{1}{\alpha_n}-\frac{1}{\alpha_{n+1}}\right)-\frac{1}{\alpha_{n+1}}}\right]^{1-\alpha^*}},
    \end{equation}
    where we have used \eqref{eq:cantor 1}, \eqref{eq:cantor 3} and executed certain quick computations. We need to examine the part in brackets of the latter denominator in order to check the convergence of the corresponding sum. Using \eqref{eq:cantor 1}, taking limits as $n\to+\infty$, and keeping in mind that $\lim_{n\to+\infty}\alpha_n=1-\alpha^*$, we may see that
    \begin{equation*}
        \lim_{n\to+\infty}\left[1-2^{1+n\left(\frac{1}{\alpha_n}-\frac{1}{\alpha_{n+1}}\right)-\frac{1}{\alpha_{n+1}}}\right]^{1-\alpha^*}=\left[1-2^{-\frac{\alpha^*}{1-\alpha^*}}\right]^{1-\alpha^*}.
    \end{equation*}
    Applying this limit on \eqref{eq:cantor 5}, we may find an absolute constant $c>0$ such that
    \begin{equation*}
        \frac{1}{\mathrm{cap}_{\alpha^*}(F)} \le c_{n_0}+c\sum_{n=n_0}^{+\infty}2^{-\sqrt{n}}<+\infty.
    \end{equation*}
    As a result, it is necessary that $\mathrm{cap}_{\alpha^*}(F)>0$ and thus $F$ is the desired set.
\end{proof}

Finally, we deal with the instance when $\alpha^*=1$, in which case the corresponding Riesz capacity turns into the more common logarithmic capacity. The next result is the counterpart of Proposition \ref{construction Cantor 2}. A counterpart of Proposition \ref{prop:cantor closed} for $\alpha^*=1$ does not require a proof, since there obviously exist Cantor sets with positive Riesz-$\alpha$ capacity, for all $\alpha\in(0,1]$.

\begin{proposition}\label{construction cantor a 1}
There exists a generalized Cantor set $F\subset \TT$ such that $\mathrm{cap}_\alpha(F)=0$, for all $\alpha\in (0,1)$, whereas $\mathrm{cap}_1(F)>0$.
\end{proposition}
\begin{proof}
Working as in the previous theorem, we will first construct a generalized Cantor set $E\subset [0,1]$. Consider the sequences
\begin{equation}\label{eq:logarithmic 1}
    d_n=\frac{1}{(n+1)^{n+1}}, \quad e_n=d_n-2d_{n+1}, \quad n\in\mathbb{N}\cup\{0\}.
\end{equation}
These two sequences define a unique generalized Cantor set on the unit circle following the procedure described before, since easily $e_n>0$, for all $n\in\mathbb{N}\cup\{0\}$. Consider the transformation $T\vcentcolon[0,1]\to\TT$ with $T(\theta)=e^{i2\pi\theta}$ and set $F=T(E)$, which is a generalized Cantor set on $\TT$. Our goal is to utilize Lemma \ref{lem:capacity of Cantor} once again. Indeed, for any $\alpha\in(0,1)$, by \eqref{eq:kernel} and \eqref{eq:logarithmic 1} we have
\begin{equation*}
    \sum_{n=0}^{+\infty}\frac{K_\alpha(d_n)}{2^{n+1}}=\sum_{n=0}^{+\infty}\frac{[(n+1)^{1-\alpha}]^{n+1}}{2^{n+1}}=\sum_{n=0}^{+\infty}\left(\frac{(n+1)^{1-\alpha}}{2}\right)^{n+1}=+\infty.
\end{equation*}
The left-hand side in Lemma \ref{lem:capacity of Cantor} implies that $\mathrm{cap}_\alpha(F)=0$, for any choice of $\alpha\in(0,1)$. We are left with evaluating $\mathrm{cap}_1(F)$. Using \eqref{eq:kernel} and \eqref{eq:logarithmic 1}, we have
\begin{align*}
    \sum_{n=0}^{+\infty}\frac{K_1(e_n)}{2^n}&=\sum_{n=0}^{+\infty}\frac{\log\frac{2(n+1)^{n+1}(n+2)^{n+2}}{(n+2)^{n+2}-2(n+1)^{n+1}}}{2^n} \\
    &\le \sum_{n=0}^{+\infty}\frac{\log2+(n+1)\log(n+1)+(n+2)\log(n+2)}{2^n}\\
    &\le \sum_{n=0}^{+\infty}\frac{3(n+2)\log(n+2)}{2^n}\\
    &<+\infty,
\end{align*}
where the first inequality follows from the relation $(n+2)^{n+2}-2(n+1)^{n+1}\ge 1$, for all $n\in\mathbb{N}\cup\{0\}$, the second is trivial, while the finiteness of the latter sum can be inferred through the root test. Via the right-hand side in Lemma \ref{lem:capacity of Cantor}, we find $\mathrm{cap}_1(F)>0$. Therefore, the set $F$ is the desired set and the proof is complete.
\end{proof}

\begin{remark}\label{remark location of cantor}
In both of the preceding propositions, we first worked on the interval $[0,1]$ and afterwards we transferred the estimations to the unit circle by a proper transformation.  Similarly, it is possible to work on any interval $[a,b]$ in order to appropriately modify the size of $E$ and its location on the unit circle. In other words, for any fixed $\alpha^*\in (0,1],$ and for any open arc $I\subset \TT,$ there exists a Cantor-type set $E\subset I$ such that $\mathrm{cap}_a(E)=0,$ for all $\alpha\in (0,\alpha^*),$ and $\mathrm{cap}_{\alpha^*}(E)>0.$
\end{remark}

\section{The first construction}
To start this section, we need one final definition. Let $E\subset \TT$ be a closed set. Recall that $E$ is characterized as a \textit{Carleson} set if
\begin{equation*}
    \int_{\TT}\log\Big(\frac{1}{\mathrm{dist}(\zeta,E)}\Big)|d\zeta|<+\infty.
\end{equation*}
It may be directly checked that every generalized Cantor set of the unit circle is a Carleson set. Therefore, the sets constructed in Propositions \ref{prop:cantor closed}, \ref{construction Cantor 2} and \ref{construction cantor a 1} are Carleson sets. For more information on Carleson sets we refer to \cite[Section 4.4]{Primer}.

Let $k\in \mathbb{N}$ and consider a Carleson set $E\subset \TT$. We may write $\TT\setminus E=\cup_{j=1}^{+\infty} I_j$, where $I_j$ is an arc on the unit circle with endpoints $a_j,b_j$. Let $\phi_E\vcentcolon\TT\to [0,+\infty)$ be the function defined by
\begin{equation}\label{eq:phi outer}
\phi_E(\zeta)=
\begin{cases}
    |(\zeta-a_j)(\zeta-b_j)|, \quad\quad\hspace{0.1cm}\text{for }\zeta\in I_j\\
    \quad 0, \hspace{3.33cm}\text{otherwise }.
\end{cases}
\end{equation}
Next, pick $N>2k$ and consider the function $f_{E,N}:\DD\rightarrow \CC$, with
\begin{equation}\label{eq: f outer}
f_{E,N}(z):=\exp\Big(\frac{N}{2\pi}\int_{\TT}\frac{\zeta+z}{\zeta-z}\log\phi_E(\zeta)|d\zeta|\Big).
\end{equation}
Following the proof of \cite[Theorem 4.4.3]{Primer}, the function $f_{E,N}$ is outer, lies in $A^k(\DD)$, and satisfies $\mathcal{Z}(f_{E,N}^*)=E$; see also \cite{Carleson}. Note that $A^k(\DD)\subset \mathcal{D}\subset \cD_\alpha$, for all $k\geq 1$ and all $\alpha\leq 1$. Hence, $f_{E,N}$ belongs to all the Dirichlet-type spaces we consider in the present work. 

The following theorem provides our first critical counterexample described in the main \hyperref[Problem]{Problem}. We shall apply Theorem \ref{thm: efkr 2} in order to inspect the cyclicity of the aforementioned function $f_{E,N}$ in the spaces $\cD_\alpha$.

\begin{theorem}\label{first}
Fix $\alpha^*\in (0,1]$. There exists a function $f\in \cD_{\alpha^*}$ which is cyclic in $\cD_\alpha$ for all $\alpha<\alpha^*$, but is not cyclic in $\cD_{\alpha^*}$. In particular, for any fixed $k\in\mathbb{N}$, the function may be chosen to belong to $A^k(\DD)$.
\end{theorem}
\begin{proof}
Let $E_{\alpha^*}\subset \TT$ be either the set constructed in Proposition \ref{construction Cantor 2} ($\alpha^*\in (0,1)$), or the set constructed in Proposition \ref{construction cantor a 1} ($\alpha^*=1$). Both sets are generalized Cantor sets and hence Carleson sets. So we may use \eqref{eq:phi outer} to define the corresponding functions $\phi_{E_{\alpha^*}}$. Fix $N>2k$. By \eqref{eq: f outer}, we have a function $f_{E_{\alpha^*},N}\in A^k(\DD)\subset\cD_{\alpha^*}$ and $\mathcal{Z}(f_{E_{\alpha^*},N}^*)=E_{\alpha^*}$. Therefore, by Propositions \ref{construction Cantor 2} and \ref{construction cantor a 1}, $\mathrm{cap}_\alpha(\mathcal{Z}(f_{E_{\alpha^*},N}^*))=0$, for all $\alpha\in (0,\alpha^*)$, whereas $\mathrm{cap}_{\alpha^*}(\mathcal{Z}(f_{E_{\alpha^*},N}^*))>0$. Hence, we may apply Theorem \ref{thm: efkr 2} to deduce cyclicity in $\cD_\alpha$, for all $\alpha\in(0,\alpha^*)$. Furthermore, by Theorem \ref{thm: efkr 1}, the positivity of $\mathrm{cap}_{\alpha^*}(\mathcal{Z}(f_{E_{\alpha^*},N}^*))>0$ yields non-cyclicity in $\cD_{\alpha^*}$, something that completes the proof.
\end{proof}

\section{The second construction}
We move on to a construction which is fundamentally different to the previous one. This second approach is inspired by and based on techniques and arguments introduced in \cite{Brown-Cohn, Carleson2} and \cite[Section 3.4]{Primer}. In these works, the authors deal with the classical Dirichlet space $\cD=\cD_1$. We are going to appropriately modify the proofs in order to produce results about the Dirichlet-type spaces $\cD_\alpha$, $\alpha\in(0,1)$, something that is partly motivated by \cite[Remark 2]{Brown-Cohn}. In our setting, the key difference is that we require uniform estimates with respect to the parameters $\alpha\in (0,1)$. Therefore, we first need to develop the necessary corresponding tools.

We start with a brief mention of another tool of potential theory with great capabilities, the harmonic measure. Even though the harmonic measure may be defined for any domain whose boundary has positive logarithmic capacity, we will only need it in the setting of the unit disc. More specifically, let $E\subseteq\TT$ be Borel. Then, the \textit{harmonic measure of} $E$ \textit{with respect to} $\DD$ is the solution of the generalized Dirichlet problem with boundary values $1$ on $E$ and $0$ on $\TT\setminus E$. By definition, the harmonic measure is a harmonic function, whose values we denote by $\omega(z,E,\DD)$, $z\in\DD$. It can be proved that
\begin{equation}\label{eq:harmonic measure}
    \omega(z,E,\DD)=\frac{1}{2\pi}\int_{E}\frac{1-|z|^2}{|z-\zeta|^2}|d\zeta|,
\end{equation}
where the integrand is the Poisson kernel. An important feature of harmonic measure that we are going to need is its following sense of continuity:
\begin{theorem}[{cf. \cite[Theorem 4.3.4]{Ransford}}]\label{thm:harmonic measure}
    Let $E\subseteq\TT$ be Borel and suppose that $\zeta$ is in the relative interior of $E$ in $\TT$. Then $\lim_{z\to\zeta}\omega(z,E,\DD)=1$.
\end{theorem}

Before delving into the main results, we prove an auxiliary lemma that will be of use in the sequel. Let us note that such an inequality is also of fundamental importance with regard to the potential theory in the finite-dimensional ball setting.

\begin{lemma}\label{Real estimation}
Fix $\alpha_1\in (0,1)$. There exists a constant $C\ge1$ depending solely on the choice of $\alpha_1$ such that
\begin{equation*}
\frac{1}{|1-z\overline{\eta}|^{1-\alpha}}\leq C \text{ }\mathrm{Re}\Big(\frac{1}{(1-z\overline{\eta})^{1-\alpha}}\Big),  
\end{equation*} 
for all $\alpha\in [\alpha_1,1]$, $z\in\DD$ and $\eta\in\TT$.
\end{lemma}
\begin{proof}
Set $z=re^{i\varphi}\in\DD$ and $\eta=e^{i\theta}\in \TT$, where $\varphi,\theta\in [0,2\pi)$ and $r\in[0,1)$. Note that
\begin{equation*}
1-z\overline{\eta}=1-re^{i(\varphi-\theta)}=1-r\cos(\varphi-\theta)-ir\sin(\varphi-\theta),
\end{equation*}
and thus
\begin{equation*}
  \mathrm{Re}(1-z\overline{\eta})=1-r\cos(\varphi-\theta)>0.  
\end{equation*}
Thus, we may write
\begin{equation}\label{arg}
\arg{(1-z\overline{\eta}})=\arctan\frac{\mathrm{Im}{(1-z\overline{\eta}})}{\mathrm{Re}{(1-z\overline{\eta}})}=\arctan\frac{-r\sin(\varphi-\theta)}{1-r\cos(\varphi-\theta)}\in\left(-\frac{\pi}{2},\frac{\pi}{2}\right).
\end{equation}
In case $\alpha=1$, the statement of the lemma is trivially true for any $C\ge1$. Next, pick $\alpha\in [\alpha_1,1)$. One may see that
\begin{equation}\label{Re}
\mathrm{Re}\Big(\frac{1}{(1-z\overline{\eta})^{1-\alpha}}\Big)=\frac{\mathrm{Re}(e^{-i(1-\alpha)\arg(1-z\overline{\eta})})}{|1-z\overline{\eta}|^{1-\alpha}}=\frac{\cos((1-\alpha)\arg(1-z\overline{\eta}))}{|1-z\overline{\eta}|^{1-\alpha}}.
\end{equation}
We wish to prove the existence of an absolute constant $C\ge1$ such that 
\begin{equation*}
   1\leq C \text{ }|1-z\overline{\eta}|^{1-\alpha}\mathrm{Re}\Big(\frac{1}{(1-z\overline{\eta})^{1-\alpha}}\Big), 
\end{equation*}
which, via \eqref{arg} and \eqref{Re} is equivalent to
\begin{equation*}
1\leq C\cos{\Big((1-\alpha) \arctan\frac{-r\sin(\varphi-\theta)}{1-r\cos(\varphi-\theta)}\Big)},
\end{equation*}
and $C\ge1$ does not depend on $r\in [0,1)$, $\varphi,\theta\in [0,2\pi)$, and $\alpha\in [\alpha_1,1)$. In particular, such a uniform constant exists if
\begin{equation}
\inf_{\alpha,r, \varphi,\theta}{\cos{\Big((1-\alpha) \arctan\frac{-r\sin(\varphi-\theta)}{1-r\cos(\varphi-\theta)}\Big)}}>0,
\end{equation}
where $ \alpha,r, \varphi,\theta$ lie in the sets mentioned above. Indeed, since the cosine and arctangent functions are even and odd, respectively, we have
\begin{equation*}
\cos{\Big((1-\alpha) \arctan\frac{-r\sin(\varphi-\theta)}{1-r\cos(\varphi-\theta)}\Big)}=\cos{\Big((1-\alpha) \arctan\frac{r|\sin(\varphi-\theta)|}{1-r\cos(\varphi-\theta)}\Big)}.
\end{equation*}
Furthermore,
\begin{equation*}
    (1-\alpha)\arctan\frac{r|\sin(\varphi-\theta)|}{1-r\cos(\varphi-\theta)}\in [0,(1-\alpha)\pi/2),
\end{equation*}
and hence,  by the monotonicity of $\cos x$ and $\arctan x$ on $[0,\pi/2)$, we obtain
\begin{equation*}
\cos{\Big((1-\alpha) \arctan\frac{r|\sin(\varphi-\theta)|}{1-r\cos(\varphi-\theta)}\Big)}\geq \cos\left((1-\alpha)\frac{\pi}{2}\right)\geq \cos\left((1-\alpha_1)\frac{\pi}{2}\right)>0,
\end{equation*}
for all $r\in [0,1)$, $\varphi,\theta\in [0,2\pi)$, and $\alpha\in [\alpha_1,1)$. This proves the assertion with $C=(\cos((1-\alpha_1)\pi/2))^{-1}>1$.
\end{proof}

Before proceeding, we recall certain standard facts relative to the Gamma function. The binomial series formula
\begin{equation}\label{eq:binomial}
    \frac{1}{(1-z)^\beta}=\sum_{k=0}^{+\infty}\frac{\Gamma(k+\beta)}{\Gamma(\beta)k!}z^k,
\end{equation}
holds for all $z\in\DD$ and all positive real numbers $\beta$, where $\Gamma$ denotes the Gamma function. Moreover, the Gamma function satisfies $\Gamma(k)=(k-1)!$, for all $k\in\mathbb{N}$. In particular, Gautschi's inequality states that
\begin{equation}\label{eq:gautschi}
    k^{1-\beta}<\frac{\Gamma(k+1)}{\Gamma(k+\beta)}<(k+1)^{1-\beta},
\end{equation}
for all $k\in\mathbb{N}\cup\{0\}$ and all $\beta\in (0,1)$; see \cite[p. 79, eq. (7)]{Gautschi}.

We are now ready to move on to the main part of this section. We will frequently work with the kernel, the potential, and the energy linked with Riesz $\alpha$-capacity; see Section \ref{sec:preliminaries}. For this reason, we will prove a handy lemma concerning the Riesz $\alpha$-energy of a Borel probability measure supported on a compact subset of $\TT$. 

Let $\mu$ be such a measure. It is well-known to experts (see, for instance, \cite[p. 566]{EFKR}) that given $\alpha\in(0,1)$, $I_\alpha[\mu]$ is comparable to $\sum_{k=0}^{+\infty}\frac{|\hat{\mu}(k)|^2}{(k+1)^\alpha}$, where
\begin{equation}\label{eq:fourier}
    \hat{\mu}(k)\vcentcolon=\int_{\TT}{\bar{\zeta}}^kd\mu(\zeta), \quad k\in\mathbb{Z},
\end{equation}
denotes $k$-th Fourier coefficient of the measure $\mu$ and the comparability constants depend on $\alpha$. Nevertheless we are going to need a slightly modified variation of this comparability. For the sake of clarity and completeness, we follow a standard procedure to provide this comparability whenever $\alpha\in (0,1)$; see \cite[Theorem 2.4.4]{Primer} for the case $\alpha=1$.

\begin{lemma}\label{lem:energy}
Fix $\alpha_1\in(0,1).$ If $\mu$ is a Borel probability measure on $\TT,$ then there exists $C>0$ depending solely on $\alpha_1$ such that
\begin{equation*}
\sum_{k=0}^{+\infty}\frac{\Gamma(k+1-\alpha)}{\Gamma(1-\alpha)k!}|\hat{\mu}(k)|^2\leq I_\alpha[\mu]\leq C\sum_{k=0}^{+\infty}\frac{\Gamma(k+1-\alpha)}{\Gamma(1-\alpha)k!}|\hat{\mu}(k)|^2,
\end{equation*}
for all $\alpha\in [\alpha_1,1).$
\end{lemma}
\begin{proof}
Let $\alpha\in[a_1,1)$. Applying Lemma~\ref{Real estimation} and relations \eqref{eq:binomial}, \eqref{eq:fourier}, we obtain
\begin{align}\label{eq:comparability2}
\notag \sum_{k=0}^{+\infty}\frac{\Gamma(k+1-\alpha)}{\Gamma(1-\alpha)k!}r^k|\hat{\mu}(k)|^2&\leq \int_{\TT}\int_{\TT}\frac{1}{|1-r\zeta\overline{\eta}|^{1-\alpha}}d\mu(\zeta)d\mu(\eta)\\
&\leq  C\sum_{k=0}^{+\infty}\frac{\Gamma(k+1-\alpha)}{\Gamma(1-\alpha)k!}r^k|\hat{\mu}(k)|^2,
\end{align}
where $C$ is the constant from Lemma \ref{Real estimation} and depends exclusively on $\alpha_1$. Thus, we have established the comparabality of the double integral in \eqref{eq:comparability2} with the series in \eqref{eq:comparability2}. 
Moreover, if $I_\alpha[\mu]<+\infty$ and $r\in(0,1),$ then
\begin{equation}\label{eq:energy 1}
\frac{2^{1-\alpha}}{(1+r)^{1-\alpha}}I_\alpha[\mu]\geq
\int_{\TT}\int_{\TT}\frac{1}{|1-r\zeta\overline{\eta}|^{1-\alpha}}d\mu(\zeta)d\mu(\eta),
\end{equation}
where the inequality $|1-w|/|1-rw|\leq 2/(1+r)$, for all $w\in \overline{\DD}$ and all $r\in (0,1)$, has been used. Combining \eqref{eq:comparability2}, \eqref{eq:energy 1} and Lebesgue's Dominated Convergence Theorem yields the left-hand side of the assertion, while Fatou's lemma provides the right-hand side. 
\end{proof}

\begin{remark}
Utilizing Gautschi's inequality \eqref{eq:gautschi} and the constant $C=(\cos((1-\alpha_1)\pi/2))^{-1}>1$ from the proof of Lemma \ref{Real estimation}, we obtain
\begin{equation*}
 |\hat{\mu}(0)|^2+\frac{1}{\Gamma(1-\alpha)}\sum_{k=1}^{+\infty}\frac{|\hat{\mu}(k)|^2}{k^\alpha}\leq I_\alpha[\mu]\leq \frac{(\cos((1-\alpha_1)\pi/2))^{-1}}{\Gamma(1-\alpha)}\sum_{k=0}^{+\infty}\frac{|\hat{\mu}(k)|^2}{(k+1)^\alpha},    
\end{equation*}
for all $\alpha\in [\alpha_1,1)$. This also yields the well known comparability in \cite[p. 566]{EFKR}.
\end{remark}

We continue with the following lemma which constructs a function belonging to prescribed Dirichlet-type spaces and satisfying certain useful properties.
\begin{lemma}\label{real}
Fix $\alpha_1\in (0,1)$. Let $E\subset \TT$ be closed, $\alpha\in[\alpha_1,1)$, and $\mu$ be a Borel probability measure supported on $E$ with $I_{\alpha}[\mu]<+\infty$. Then, the function $f_{\alpha,\mu}\vcentcolon\DD\to \CC$ defined via the formula
\begin{equation}\label{eq:main lemma 1}
f_{\alpha,\mu}(z)\vcentcolon=\int_{\TT}\frac{1}{(1-z\overline{\eta})^{1-\alpha}}d\mu(\eta), \quad z\in \DD,
\end{equation}
is well defined and enjoys the following properties:
\begin{enumerate}
    \item[\textup{$(\alpha)$}] $f_{\alpha,\mu}\in \cD_{\alpha}$ and $\|f_{\alpha,\mu}\|_{\alpha}\leq C I_{\alpha}[\mu]^{1/2}$, where $C\ge1$ is a constant that does not depend on $\alpha\in [\alpha_1,1)$;
    \item[\textup{$(\beta)$}] $|f_{\alpha,\mu}(z)|\leq 1/\textup{dist}(z,E)^{1-\alpha}$, for all $z\in \DD$;
    \item[\textup{$(\gamma)$}] $\textrm{Re}f_{\alpha,\mu}(z)> 0$, for all $z\in \DD;$
    \item[\textup{$(\delta)$}] there exists a constant $C\ge1$, not depending on $\alpha\in[\alpha_1,1)$, such that $\textrm{Re}f^{*}_{\alpha,\mu}\geq \frac{1}{C}K_{\alpha,\mu}$, $\alpha$-almost everywhere in $\TT$;
    \item[\textup{$(\epsilon)$}] $\textrm{Re}f^{*}_{\alpha,\mu}\in L^1(\TT)$.
\end{enumerate}
\end{lemma}
\begin{proof}
\textup{($\alpha$)} The function $f_{\alpha,\mu}$ is holomorphic in $\DD$ with power series expansion given by
\begin{equation}\label{eq:main lemma 2}
    f_{\alpha,\mu}(z)=\sum_{k=0}^{+\infty}\frac{\Gamma(k+1-\alpha)}{\Gamma(1-\alpha)k!}\hat{\mu}(k)z^k, \quad z\in \DD,
\end{equation}
something that can be computed by combining \eqref{eq:binomial} with \eqref{eq:main lemma 1} and executing simple calculations. Hence, by \eqref{eq:norm} and \eqref{eq:main lemma 2},
\begin{align*}
\|f_{\alpha,\mu}\|_\alpha^2 &= \sum_{k=0}^{+\infty}(k+1)^\alpha\frac{\Gamma(k+1-\alpha)^2}{\Gamma(1-\alpha)^2(k!)^2}|\hat{\mu}(k)|^2 \\
&=1+\frac{1}{\Gamma(1-\alpha)}\sum_{k=1}^{+\infty}\frac{\Gamma(k+1-\alpha)}{\Gamma(1-\alpha)k!}\frac{(k+1)^\alpha\Gamma(k+1-\alpha)}{k!}|\hat{\mu}(k)|^2\\
&\le 1+\frac{1}{\Gamma(1-\alpha)}\sum_{k=1}^{+\infty}\frac{\Gamma(k+1-\alpha)}{\Gamma(1-\alpha)k!}\frac{(k+1)^\alpha\Gamma(k+1)}{k^\alpha k!}|\hat{\mu}(k)|^2\\
&\le 1+\frac{2^\alpha}{\Gamma(1-\alpha)}\sum_{k=1}^{+\infty}\frac{\Gamma(k+1-\alpha)}{\Gamma(1-\alpha)k!}|\hat{\mu}(k)|^2\\
&\le \max\left\{1,\frac{2^\alpha}{\Gamma(1-\alpha)}\right\}I_\alpha[\mu]\\
&\le \max\left\{1,\frac{2}{\Gamma(1-\alpha_1)}\right\}I_\alpha[\mu],
\end{align*}
where in the second line we just did some recombinations, in the third we applied Gautschi's inequality, in the fourth we made use of the facts that $k+1\le 2k$ and $\Gamma(k+1)=k!$, for all $k\in\mathbb{N}$, in the fifth we applied Lemma \ref{lem:energy}, and in the last we made use of the fact that the Gamma function is decreasing in $(0,1),$ and hence, $\Gamma(1-\alpha)\ge\Gamma(1-\alpha_1),$ whenever $\alpha\in[\alpha_1,1)$. Setting $C=\max\{1,\sqrt{2/\Gamma(1-\alpha_1)}\}$, we obtain the desired absolute constant.\\[2pt]
\textup{($\beta$)} Since $E$ is the compact support of $\mu$, integration with respect to $\mu$ outside of $E$ yields zero. Thus, returning to \eqref{eq:main lemma 1}, we get
\begin{equation*}
|f_{\alpha,\mu}(z)|\leq\int_{E}\frac{1}{|1-z\overline{\eta}|^{1-\alpha}}d\mu(\eta)
\leq \int_{E}\frac{1}{\textrm{dist}(z,E)^{1-\alpha}}d\mu(\eta)
=\frac{1}{\textrm{dist}(z,E)^{1-\alpha}},
\end{equation*}
for all $z\in \DD$, since $\mu$ is a probability measure and so $\mu(E)=1$.\\[2pt]
\textup{($\gamma$)}  Lemma \ref{Real estimation}, in conjunction with \eqref{eq:main lemma 1}, implies
\begin{equation*}
  \mathrm{Re}f_{\alpha,\mu}(z)=\int_{\TT}\mathrm{Re}\Big(\frac{1}{(1-z\overline{\eta})^{1-\alpha}}\Big)d\mu(\eta)\geq\int_{\TT}\frac{1}{C}\frac{1}{|1-z\overline{\eta}|^{1-\alpha}}d\mu(\eta)> 0,  
\end{equation*}
for all $z\in \DD$.\\[2pt]
\textup{($\delta$)} Since $f_{\alpha,\mu}\in \cD_a$, the radial limit function $f^*_{\alpha,\mu}$ exists in $\TT$, except possibly on a set having Riesz $\alpha$-capacity zero. Pick a point $\zeta \in\TT$ such that $f^*_{\alpha,\mu}(\zeta)$ exists. By Fatou's Lemma, Lemma \ref{Real estimation} and \eqref{eq:potential}, we obtain
\begin{align*}
\textrm{Re}f^*_{\alpha,\mu}(\zeta)&\geq 
\int_E\liminf_{z\rightarrow \zeta}\textrm{Re}\Big(\frac{1}{(1-z\overline{\eta})^{1-\alpha}}\Big)d\mu(\eta)\\
&\geq \frac{1}{C}\int_{E}\liminf_{z\rightarrow\zeta}\frac{1}{|1-z\overline{\eta}|^{1-\alpha}}d\mu(\eta)\\
&=\frac{1}{C}K_{\alpha,\mu}(\zeta).
\end{align*}
Note that by the proof of Lemma \ref{Real estimation}, the constant $C$ can be chosen to be $(\cos((1-\alpha_1)\pi/2))^{-1}$, and hence its value does not depend on $\alpha\in[\alpha_1,1)$, but merely on the initially fixed $\alpha_1$.\\[2pt]
\textup{($\epsilon$)} By Fatou's Lemma and Fubini's Theorem, we have
\begin{align*}
\int_{\TT}\textrm{Re}f^*_{\alpha,\mu}(\zeta)|d\zeta|&\leq \liminf_{r\rightarrow1^-} \int_{\TT}\textrm{Re}f_{\alpha,\mu}(r\zeta)|d\zeta|
\leq  \liminf_{r\rightarrow1^-} \int_{\TT}\int_{\TT}\frac{1}{|1-r\zeta\overline{\eta}|^{1-\alpha}}|d\zeta|d\mu(\eta),
\end{align*}
and the last limit inferior is finite because $\alpha\in (0,1)$; see \cite[Proposition 1.4.10]{Rudin}. Thus, we infer that $\textrm{Re}f^*_{\alpha,\mu}\in L^1(\TT)$.
\end{proof}

With all the above lemmas and tools in hand, we are able to build the desired counterexample.

\begin{theorem}\label{construction of function1}
Let $E\subset \TT$ be a closed set for which we may find a strictly increasing sequence $\{\alpha_n\}_{n\in \mathbb N}\subset (0,1)$ such that $\lim_{n\to+\infty}\alpha_n=1$ and $\mathrm{cap}_{\alpha_n}(E)=0$, for all $n\in \mathbb{N}$. Then, there exists $F\in \cD_\alpha$, for all $\alpha<1$, such that $\lim_{z\rightarrow \zeta}\mathrm{Re}F(z)=+\infty$, for all $\zeta\in E$. The function $F$ may be chosen to be continuous on $\overline{\DD}\setminus E$.
\end{theorem}
\begin{proof}
By assumption $\mathrm{cap}_{\alpha_n}(E)=0$. Therefore, we may create a sequence $\{E_n\}$ of compact subsets of $\TT$ such that $\mathrm{cap}_{\alpha_n}(E_n)\in(0,1)$, for all $n\in \mathbb{N}$, and
\begin{equation}\label{series cap 1/2}
\sum_{n=1}^{+\infty}\mathrm{cap}_{\alpha_n}(E_n)\le\sum_{n=1}^{+\infty}\mathrm{cap}_{\alpha_n}(E_n)^{1/2}<+\infty.
\end{equation}
Each $E_n$ can be chosen so that it is a neighborhood of $E$ in the topology of $\TT$, that is, each point of $E$ belongs to the non-empty interior of every $E_n$, $n\in\mathbb{N}$, with respect to the topology of the unit circle. Let $\nu_n$ be an equilibrium measure of $E_n$ with respect to $\mathrm{cap}_{\alpha_n}(E_n)>0$. In particular, the potential of $\nu_n$ satisfies 
\begin{equation}\label{K I cap}
K_{\alpha_n,\nu_n}=I_{\alpha_n}[\nu_n]=\frac{1}{\mathrm{cap}_{\alpha_n}(E_n)},   
\end{equation}
$\alpha_n$-almost everywhere on $E_n$; see Section \ref{sec:preliminaries}. Following the notation in Lemma \ref{real}, the functions $f_{\alpha_n,\nu_n}$ exist and $f_{\alpha_n,\nu_n}\in \Hol$. Thus, we may consider the function
\begin{equation*}
 f(z)\vcentcolon=\sum_{n=0}^{+\infty}\mathrm{cap}_{\alpha_n}(E_n)f_{\alpha_n,\nu_n}(z), \quad z\in \DD.   
\end{equation*}
Lemma \ref{real}\textup{($\beta$)} implies 
\begin{equation*}
    |f_{\alpha_n,\nu_n}(z)|\le \frac{1}{\mathrm{dist}(z,E_n)^{1-\alpha_n}}\le \frac{1}{(1-|z|)^{1-\alpha_n}}\le \frac{1}{(1-|z|)^{1-\alpha_1}},
\end{equation*}
for all $n\in\mathbb{N}$. Combining this with \eqref{series cap 1/2} and using the Weierstrass M-test, we see that the partial sums of $f$ converge locally uniformly to the well-defined function $f$. Then, Weierstrass' Theorem certifies that $f\in \Hol$. On another note, $\mathrm{Re}f_{\alpha_n,\nu_n}$ is positive (see Lemma \ref{real}\textup{($\gamma$)}) and harmonic in the unit disc. Ergo, it may be written as its Poisson integral, that is,
\begin{equation*}
\mathrm{Re}f_{\alpha_n,\nu_n}(z)=\frac{1}{2\pi}\int_{\TT}\frac{1-|z|^2}{|z-\eta|^2}\mathrm{Re}f^*_{\alpha_n,\nu_n}(\eta)|d\eta|.
\end{equation*} 
However, by Lemma \ref{real}\textup{($\delta$)}, we have
\begin{equation*}
\mathrm{Re}f^*_{\alpha_n,\nu_{n}}\geq \frac{1}{C}K_{\alpha_n,\nu_{n}},
\end{equation*}
$\alpha_n$-almost everywhere on $\TT$, where $C=\cos(((1-\alpha_1)\pi/2))^{-1}$ does not depend on $n\in\mathbb{N}$. As a result,
\begin{equation}\label{eq:poisson}
\mathrm{Re}f_{\alpha_n,\nu_n}(z)\ge \frac{1}{2\pi}\int_{\TT}\frac{1-|z|^2}{|z-\eta|^2}\frac{K_{\alpha_n,\nu_n}(\eta)}{C}|d\eta| \ge \frac{1}{2\pi}\int_{E_n}\frac{1-|z|^2}{|z-\eta|^2}\frac{K_{\alpha_n,\nu_n}(\eta)}{C}|d\eta|.
\end{equation}
Applying \eqref{K I cap} on \eqref{eq:poisson} yields
\begin{equation*}
\mathrm{Re}f_{\alpha_n,\nu_n}(z)\ge\frac{1}{C\mathrm{cap}_{\alpha_n}(E_n)}\frac{1}{2\pi}\int_{E_n}\frac{1-|z|^2}{|z-\eta|^2}|d\eta|.  \end{equation*}
Nevertheless, the integral at the right is exactly the harmonic measure of $E_n$ with respect to the unit disc evaluated at $z$, i.e. the quantity $\omega(z,E_n,\DD)$; see \eqref{eq:harmonic measure}. In particular, we have
\begin{equation*}
 \mathrm{Re}f_{\alpha_n,\nu_n}(z)\ge \frac{\omega(z,E_n,\DD)}{C\mathrm{cap}_{\alpha_n}(E_n)},   
\end{equation*}
for all $z\in\DD$. Fix $\zeta\in E$ and take limit inferiors to obtain
\begin{equation*}
   \liminf_{z\to\zeta}\mathrm{Re}f_{\alpha_n,\nu_n}(z)\ge \frac{1}{C\mathrm{cap}_{\alpha_n}(E_n)}\liminf_{z\to\zeta}\omega(z,E_n,\DD). 
\end{equation*}
Recall that the interior of each $E_n$ contains $E$, and hence, the fixed point $\zeta$ lies in the interior of every $E_n$. Therefore, Theorem \ref{thm:harmonic measure} yields $\lim_{z\to\zeta}\omega(z,E_n,\DD)=1$, for all $n\in\mathbb{N}$. Consequently,
\begin{equation}\label{AAA}
\liminf_{z\to\zeta}\mathrm{Re}f_{\alpha_n,\nu_n}(z)\ge\frac{1}{C\mathrm{cap}_{\alpha_n}(E_n)}, \quad\text{for all }\zeta\in E \text{ and all }n\in\mathbb{N}.
\end{equation}
Fix $N\in \mathbb{N}$. Then, by the definition of $f$,
\begin{equation*}
  \liminf_{z\rightarrow \zeta}\mathrm{Re}f(z)\geq \sum_{n=1}^N \mathrm{cap}_{\alpha_n}(E_n)\liminf_{z\rightarrow \zeta}\mathrm{Re}(f_{\alpha_n,\nu_{n}})(z)\geq \frac{N}{C},  
\end{equation*}
which in turn implies that
\begin{equation*}
  \lim_{z\rightarrow \zeta}\mathrm{Re}f(z)=+\infty, \quad \text{for all }\zeta\in E.  
\end{equation*}
So, $f$ satisfies one of the required conditions. Our next objective is to modify $f$ accordingly, so that it is rendered continuous on $\overline{\DD}\setminus E$. Pick an increasing sequence $\{r_n\}\subset(0,1)$ with $\lim_{n\to+\infty}r_n=1$ and set 
\begin{equation*}
f_n(z):=f_{\alpha_n,\nu_n}(r_nz), \quad n\in\mathbb{N}.    
\end{equation*}
By \eqref{AAA}, we may beforehand choose $\{r_n\}$, so that
\begin{equation*}
\mathrm{Re}f_{\alpha_n,\nu_n}(r_n\zeta)\geq \frac{1}{2C\mathrm{cap}_{\alpha_n}(E_n)},
\end{equation*}
for all $\zeta \in E$ and all $n\in\mathbb{N}$, due to the compactness of $E$. Finally, set 
\begin{equation*}
   F(z)\vcentcolon=\sum_{n=0}^{+\infty}\mathrm{cap}_{\alpha_n}(E_n)f_{n}(z), \quad z\in \DD. 
\end{equation*}
Exactly as before, $F\in \Hol$, it is continuous on $\overline{\DD}\setminus E$, and
\begin{equation*}
\lim_{z\rightarrow \zeta}\mathrm{Re}F(z)=+\infty,    
\end{equation*}
for all $\zeta\in E$. It remains to prove that $F\in \cD_\alpha$, for all $\alpha\in (\alpha_1,1)$. Note that $f_n\in \mathcal{O}(\overline{\DD})$, and thus, $f_n$ and the partial sums of $F$ lie in $\DDD_\alpha$, for all $\alpha\in \RR$. Pick $\alpha\in (\alpha_1,1)$. Then, there exists $n_0\in \mathbb N$ such that $\alpha_n>\alpha$ for all $n> n_0$. In light of the arguments in Lemma\ref{real}\textup{($\alpha$)}, we obtain that $f_n\in\mathcal{D}_{\alpha_n}$ and in particular
\begin{equation}\label{PPP}
\|f_n\|_{\alpha_n}\leq CI_{\alpha_n}[\nu_n]^{1/2}=C\mathrm{cap}_{\alpha_n}(E_n)^{-1/2},
\end{equation}
where the equality is due to the definition of the equilibrium measure, and $C>0$ does not depend on the sequence $(\alpha_n)_{n\geq n_0}$. Indeed, the constant may be chosen equal to $\max\{1,\sqrt{2/\Gamma(1-\alpha_1)}\}$. Hence, we obtain
\begin{align*}
\|F\|_\alpha&\leq\liminf_{N\rightarrow +\infty}\left|\left|\sum_{n=0}^{N}\mathrm{cap}_{\alpha_n}(E_n)f_{n}\right|\right|_\alpha\\
&\leq \liminf_{N\rightarrow +\infty}\left[\sum_{n=0}^{n_0}\mathrm{cap}_{\alpha_n}(E_n)\|f_{n}\|_\alpha +\sum_{n=n_0+1}^{N}\mathrm{cap}_{\alpha_n}(E_n)\|f_{n}\|_\alpha\right]\\
&\leq\sum_{n=0}^{n_0}\mathrm{cap}_{\alpha_n}(E_n)\|f_{n}\|_\alpha +\liminf_{N\rightarrow +\infty}\sum_{n=n_0+1}^{N}\mathrm{cap}_{\alpha_n}(E_n)\|f_{n}\|_{\alpha_n}\\
&\leq\sum_{n=0}^{n_0}\mathrm{cap}_{\alpha_n}(E_n)\|f_{n}\|_\alpha +C\liminf_{N\rightarrow +\infty}\sum_{n=n_0+1}^{N}\mathrm{cap}_{\alpha_n}(E_n)^{1/2}\\
&<+\infty,
\end{align*}
where we have consecutively used that $\|f_n\|_\alpha\le \|f_n\|_{\alpha_n}$, for all $n>n_0$, relation \eqref{PPP}, the fact that each $f_n$ belongs to $\cD_\alpha$, and finally relation \eqref{series cap 1/2}. Thus, $F\in\cD_{\alpha}$. Since the choice of $\alpha\in(\alpha_1,1)$ was arbitrary, we infer that $F\in\cD_\alpha$ for all $\alpha\in(\alpha_1,1)$. Finally, as we mentioned in Subsection \ref{subsec:dirichlet}, Dirichlet-type spaces form a decreasing family, and hence $F\in\cD_\alpha$ for all $\alpha<1$. Summing up, the function $F$ has all the desired properties and we are done.
\end{proof}

The theorem below provides the second critical counterexample described in the main \hyperref[Problem]{Problem} for $\alpha^*=1$. Its proof follows in the footsteps of \cite[Theorem A]{Brown-Cohn} and \cite[Theorem 4.3.2, Theorem 9.2.7]{Primer}. 

\begin{theorem}\label{thm:second construction}
    There exists a function $f\in\cD_\alpha\cap A(\DD)$, for all $\alpha<1$, which is cyclic in $\cD_\alpha$, for all $\alpha<1$, but cannot be cyclic in $\cD$.
\end{theorem}
\begin{proof}
    By Proposition \ref{construction cantor a 1}, there exists a generalized Cantor set (and by extension a Carleson set) $E\subset\TT$ such that $\mathrm{cap}_\alpha(E)=0$, for all $\alpha\in (0,1)$, but $\mathrm{cap}_1(E)>0$. Then, Theorem \ref{construction of function1} is applicable and there exists $F\in \cD_\alpha\cap C(\overline{\DD}\setminus E)$ for all $\alpha<1$, such that $\lim_{z\rightarrow \zeta}\mathrm{Re}F(z)=+\infty$ for all $\zeta\in E$. The function $f\vcentcolon=e^{-F}$ clearly belongs to $A(\DD)$ and satisfies $\mathcal{Z}(f^*)=E$. Since $f'=-fF'$ with $F\in\cD_\alpha$, for all $\alpha<1$, while $f$ remains bounded, the integral in \eqref{eq:integral} provides at once the membership of $f$ in every Dirichlet-type space $\cD_\alpha$, $\alpha<1$. Revisiting the proof of Theorem \ref{construction of function1}, we recall that $f=\exp(-\sum_{n=0}^{+\infty}f_n)$, where the sum contains functions $f_n\in\mathcal{O}(\overline{\DD})$ with $\mathrm{Re}f_n\ge0$. In particular, $\sum_{n=0}^{+\infty}\|f_n\|_\alpha<+\infty$, for all $\alpha<1$. Set $F_N\vcentcolon=\exp(-\sum_{n=N}^{+\infty}f_n)$, $N\in\mathbb{N}$. Then $F_N\to 1$ pointwise and clearly $\sup_{N\in\mathbb{N}}\|F_N\|_\alpha<+\infty$, for all $\alpha<1$. In view of \cite[p.272]{Brown-Shields}, this implies that for any choice of $\alpha<1$,
    \begin{equation}\label{N}
    F_N\rightarrow1 \quad  \text{weakly in }\cD_\alpha.
    \end{equation}
    Furthermore, $F_N/f=\exp(\sum_{n=0}^{N-1}f_n)$ and thus $F_N/f\in\mathcal{O}(\overline{\DD})$, for any $N\in\mathbb{N}$. This signifies that each function $F_N/f$ acts as a multiplier in every space $\cD_\alpha$, $\alpha\in \RR$; see \cite[p.273]{Brown-Shields}. Therefore $F_N=(F_N/f)f\in[f]_\alpha$, which is a sequentially weakly closed subspace of $\cD_\alpha$ (since $[f]_\alpha$ is a closed convex subset). By \eqref{N} we understand that $1\in[f]_\alpha$ and hence $f$ is cyclic in $\cD_\alpha$, for all $\alpha<1$, by Remark \ref{equiv cyclic}. If $f$ does not belong in $\cD$, we are done. Finally, if $f$ actually belongs to $\cD$, the positivity of the logarithmic capacity $\mathrm{cap}_1(\mathcal{Z}(f^*))=\mathrm{cap}_1(E)$ yields that $f$ cannot be cyclic in $\cD$ in view of Theorem \ref{thm: efkr 1}.
\end{proof}

\begin{remark}\label{r}
One might wonder about the actual membership of $f$ in the Dirichlet space $\mathcal{D}$. Suppose that $f\in \mathcal{D}$. If $F=-\log f\in \mathcal{D}$, then $f$ must be cyclic in $\mathcal{D}$, see \cite[Theorem 4.4]{Extrem polyn}. As we mentioned, this cannot be true due to the positivity of the logarithmic capacity of $\mathcal{Z}(f^*)$. Thus, either $f\in \mathcal{D}$ and $F\notin\mathcal{{D}}$, or $f\notin\mathcal{D}$.
\end{remark}

Finally, we may employ the same scheme for the parameters $\alpha^*\in(0,1)$ in order to prove the following theorem whose proof we omit for the sake of avoiding repetition:

\begin{theorem}\label{thm:second construction for a}
Fix $\alpha^*\in (0,1)$. There exists a function $f\in \cD_\alpha\cap A(\DD)$, for all $\alpha<\alpha^*$, which is cyclic in $\cD_\alpha$, for all $\alpha<\alpha^*$, but cannot be cyclic in $\DDD_{\alpha^*}$.
\end{theorem}

The function in the theorem above serves as a second critical counterexample in the main \hyperref[Problem]{Problem} for $\alpha^*\in (0,1)$.

\section{Further Questions and Remarks}
There exists a function $f\in \DDD_\alpha$, for all $\alpha<1$, but $f\notin\DDD$. Indeed, such a function is given simply by $f(z)=\sum_{k=0}^{+\infty}\frac{1}{k}z^k$. In light of the results obtained in Theorems \ref{first}, \ref{thm:second construction} and \ref{thm:second construction for a}, we may formulate another question.

\begin{question*}
Fix $\alpha^*\in (0,1]$. Does there exist a function $f\in \DDD_\alpha$, $\alpha\in (0,\alpha^*)$, which is cyclic for all $\alpha<\alpha^*$, but $f\notin\DDD_{\alpha^*}$?  
\end{question*}

On another note, over the last decade, authors have also investigated cyclicity theory in several variables. A natural question would be whether the procedure described in this work can be extended in the setting of several complex variables. 

More specifically, one might wonder about the construction of generalized Cantor sets in the unit ball or polydisc of $\CC^n$, $n\ge2$, satisfying certain Riesz $\alpha$-capacity properties for prescribed indices $\alpha$. Then, these generalized Cantor sets could be turned into boundary zero sets of a function and the relation between the cyclicity of the function and the positivity of the Riesz $\alpha$-capacity of the set could be studied; see \cite{Chalmoukis1, Chalmoukis, Vavitsas1,  Vavitsas-Zarvalis} for related results involving Riesz $\alpha$-capacity and polynomials. 

Let us note that Cantor-type sets, and especially corner-like Cantor sets in $\RR^n$, are of great importance in estimating fractional Lipschitz caloric capacities related to fractional heat equations. For instance, we refer the interested reader to \cite{Hernandez1, Hernandez2} for recent work on this topic, as well as constructions and capacity estimates of hyper Cantor sets. Consequently, these works provide further motivation for extending the present framework to the multivariable setting.

\section*{Acknowledgements}
We thank \L. Kosi\'{n}ski and T. Ransford for bringing the main \hyperref[Problem]{Problem} to our attention and for the helfpul discussion on its solution. This work was supported by National Key R\&D Program of China, No. 2024YFA1015200 and the Natural Science Foundation of Guangdong Province, No. 2025A1515011428.

\bibliographystyle{plain}

\begin{thebibliography}{99}
\bibitem{Aleman1} A. Aleman, Hilbert Spaces of Analytic Functions Between the Hardy and the Dirichlet Space, \emph{Proc. Amer. Math. Soc.}, \textbf{115}(1): 97--104, 1992


\bibitem{Extrem polyn} C. B\'{e}n\'{e}teau, A.Condori, C. Liaw, D. Seco, and A. Sola, Cyclicity in Dirichlet-type spaces and extremal polynomials, \textit{J. Anal. Math.}, \textbf{126}: 259--286, 2015

\bibitem{Kosinski-Sola1} C. B\'{e}n\'{e}teau, G. Knese, \L. Kosi\'{n}ski, C. Liaw, D. Seco and A. Sola, Cyclic polynomials in two variables, \textit{Trans. Amer. Math. Soc.}, \textbf{368}(12): 8737--8754, 2016


\bibitem{Beurling} A. Beurling, On two problems concerning linear transformations in Hilbert space, \textit{Acta Math.}, \textbf{81}: 239--255, 1948

\bibitem{Bergvist}
L. Bergqvist, A note on cyclic polynomials in polydiscs, \textit{Anal. Math. Phys.}, \textbf{8}: 197--211, 2018

\bibitem{Brown-Cohn} L. Brown and W. Cohn, Some Examples of Cyclic Vectors in the Dirichlet Space, \textit{Proc. Amer. Math. Soc.}, \textbf{95}(1): 42--46, 1985 

\bibitem{Brown-Shields} L. Brown and A. L. Shields, Cyclic vectors in the Dirichlet space, \textit{Trans. Amer. Math. Soc.}, \textbf{285}(1): 269--303, 1984

\bibitem{Carleson} L. Carleson, Selected Problems on Exceptional Sets, \emph{Van Nostrand, Princeton, N.J.}, 1967

\bibitem{Carleson2} L. Carleson, Sets of uniqueness for functions regular in the unit circle,
\emph{Acta Math.}, \textbf{87}: 325--345, 1952

\bibitem{Chalmoukis1}
N. Chalmoukis and M. Hartz, Potential theory and boundary behavior in the Drury-Arveson space, preprint, available at arXiv:2410.07773 [math.FA]

\bibitem{Chalmoukis}
N. Chalmoukis and M. Hartz, Totally null sets and capacity in Dirichlet type spaces, \emph{J. London Math. Soc.}, \textbf{106}: 2030--2049, 2022


\bibitem{Primer} O. El-Fallah, K. Kellay, J. Mashreghi and T. Ransford, A Primer on the Dirichlet Space, \emph{Cambridge Tracts in Mathematics 203, Cambridge University Press, Cambridge}, 2014

\bibitem{EFKR} O. El-Fallah, K. Kellay, and T. Ransford, Cantor sets and cyclicity in weighted Dirichlet spaces, \emph{J. Math. Anal. Appl.}, \textbf{372}(2): 565--573, 2010

\bibitem{EFKR2} O. El-Fallah, K. Kellay, and T. Ransford, Cyclicity in the Dirichlet space, \textit{Ark. Mat.}, \textbf{44}: 61--86, 2006

\bibitem{EFKR*} O. El-Fallah, K. Kellay, and T. Ransford, On the Brown-Shields conjecture
for cyclicity in the Dirichlet space, \emph{Adv. Math.}, \textbf{222}(6): 2196--2214, 2009

\bibitem{Frostman} O. Frostman, Potentiel d’        \'{e}quilibre et capacit\'{e} des ensembles avec
quelques applications \`{a} la th\'{e}orie des fonctions, \emph{Thesis Meddel. Lunds Univ.
Mat. Sem.}, \textbf{3}: 1--118, 1935

\bibitem{Gautschi} W. Gautschi, Some Elementary Inequalities Relating to the Gamma and Incomplete Gamma Function, \textit{Journal of Mathematics and Physics}, \textbf{38}(1-4):  77--81, 1959

\bibitem{Hernandez1}
J. Hernández, The fractional Lipschitz caloric capacity of Cantor sets, \emph{J. London Math. Soc.}, \textbf{113}:  e70493, 2026

\bibitem{Hernandez2}
J. Hernández, Removable singularities for Lipschitz fractional caloric functions in time varying domains, preprint, available at arXiv:2412.18402 [math.AP]


\bibitem{Kosinski-Sola2} 
G. Knese, \L. Kosi\'{n}ski,  T. Ransford and A. Sola, Cyclic polynomials in anisotropic Dirichlet spaces, \emph{J. Anal. Math.}, \textbf{138}(1): 23--47, 2019

\bibitem{Vavitsas1} 
\L. Kosi\'{n}ski and D. Vavitsas, Cyclic polynomials in Dirichlet-type spaces in the unit ball of $\CC^2$, \textit{Constr. Approx.}, \textbf{58}(2): 343--361, 2023

\bibitem{Mironov}
M. Mironov and J. Sampat, Jointly cyclic polynomials and maximal domains, \emph{Proc. Amer. Math. Soc.}, \textbf{153}(11):4781--4795, 2025

\bibitem{Ransford}  
T. Ransford, Potential Theory in the Complex Plane, \textit{London Mathematical Society Student Texts, Cambridge University Press, Cambridge}, \textbf{28}, 1995

\bibitem{Richter1} 
S. Richter, Invariant subspaces in Banach spaces of analytic functions, \emph{Trans. Amer. Math. Soc.}, \textbf{304}: 585--616, 1987

\bibitem{Richter-Shields} 
S. Richter and A. Shields, Bounded analytic functions in the Dirichlet space, \emph{Math. Z.} \textbf{198}: 151--159, 1988

\bibitem{Richter} 
S. Richter and J. Sunkes, Hankel operators, invariant subspaces, and cyclic vectors in the Drury-Arveson space, \textit{Proc. Amer. Math. Soc.}, \textbf{144}(6): 2575--2586, 2016

\bibitem{Ross} 
W. T. Ross, The classical Dirichlet space, in Recent advances in operator-related function theory, \emph{Contemp. Math.}  \textbf{393}: 171--197, 2016

\bibitem{Rudin} 
W. Rudin, Function Theory in the Unit Ball of $\CC^n$, \emph{Grundlehren  der Mathematischen Wissenschaften 241, Springer-Verlag, New York-Berlin}, 1980

\bibitem{Sampat}
J. Sampat, Cyclicity preserving operators on spaces of analytic functions in $\CC^n$, \emph{Integral Equations Operator Theory} \textbf{93}(2), Paper No. 14, 20 pp., 2021

\bibitem{Taylor} 
G. D. Taylor, Multipliers on $D_\alpha,$ \emph{Trans. Amer. Math. Soc.}, \textbf{123}: 229--240, 1966

\bibitem{Pouriya} 
P. T. Ziarati, Examples of critically cyclic functions in the Dirichlet spaces of the ball, preprint, available at arXiv: 2601.08651 [math.CV]

\bibitem{Vavitsas2}
D. Vavitsas, A note on cyclic vectors in Dirichlet-type spaces in the unit ball of ${\mathbb C}^n$, \textit{Canad. Math. Bull.},  \textbf{66}(3): 886--902, 2023 

\bibitem{Vavitsas-Zarvalis} 
D. Vavitsas and K. Zarvalis, Non-cyclicity and polynomials in Dirichlet-type spaces of the unit ball, \emph{Bull. London Math. Soc.}, \textbf{56}(120): 3905--3919, 2024

\end{thebibliography}

\end{document}